\documentclass[12pt]{amsart}

\usepackage{amsmath}%
\usepackage{amsfonts}%
\usepackage{amssymb}%
\usepackage{graphicx}
\newtheorem{theorem}{Theorem}

\newtheorem{corollary}[theorem]{Corollary}

\newtheorem{lemma}[theorem]{Lemma}

\newtheorem{remark}[theorem]{Remark}

\begin{document}

\title{A Characterization Based On Products of Order Statistics}
\author{G. Arslan 
\\Izmir University of Economics, Turkey}

\begin{abstract}
A new characterization for power function distributions is obtained which is based on products of order statistics. This result may be considered as a generalization of some recent results for contractions. We note that in this result the product consists of order statistics from independent samples. The new result is also related to some scheme of ranked set sampling.
\end{abstract}

\maketitle
 
\section{Introduction}
\label{Intro}

In this paper, a new characterization, involving a distributional relation of products of independent 
order statistics is given. This result may be considered as an extension of a particular type of characterization with contraction, where the contraction is obtained via a power distribution. Such 
type of characterizations have been studied, recently by Weso{\l}owski and Ahsanullah (2004), Oncel et al. (2005), and Beutner and Kamps (2008), among others.  As will be seen later, this new characterization 
also will provide a relational property of the power distribution. Another useful property of the new characterization is that it can be actually related to a special scheme of ranked set samples. 

Consider three independent random variables $X$, $Y$, and $U$, where $U$ has some known distribution. 
There are several recent results for relations of the type 

	\begin{equation} \label{basic}
		X \stackrel{d}= YU 
	\end{equation}
	
In its most basic form $U$ can be assumed to have a uniform distribution concentrated on (0,1). In this case, relation (\ref{basic}) is an example of a contraction. These type of relations have some applications like in economic modelling and reliability, for example. Some of the first results of this type were obtained, among others, by Alamatsaz (1985), Kotz and Steutel (1988), and Alzaid and Al-Osh (1991). 

There are many interesting distributional relations and characterization results based on this type of relation. In this paper, we will consider relations where $X$ and $Y$ will be some order statistic. 

We will write $X \sim Pow(\alpha)$ if $F_X(x) = x^{\alpha}, \; \alpha > 0, \; x \in (0,1)$, and $Y \sim Par(\alpha)$, if $F_Y(y) = 1-y^{-\alpha}, \; \alpha > 0, \; y \in (1, \infty)$.

As will be seen later, it will be convenient to describe the new results in terms of some special type 
of a ranked set sample scheme. Actually the result in this paper may be interpreted in different ways. Since 
we will use ordinary order statistics and generalized order statistics we just give the definition of generalized order statistics. 

Let $F$ be an absolutely continuous distribution function (df) with density function $f$. Let $n \in \mathbb{N}$, 
$\tilde{m} = (m_1, \; \ldots, \; m_{n-1}) \in \mathbb{R}^{n-1}$, $k>0$, and for all $1 \leq i \leq n-1$,
\[ \gamma_i = k+n-1+M_i > 0, \]
\noindent where $M_i = \sum^{n-1}_{j=i} m_j$. The random variable $X \left( r,n,\tilde{m},k \right)$, 
$1 \leq r \leq n$, is called the $r$-th generalized order statistics from an absolutely continuous df $F$ 
if their joint density function is 
\[
\begin{split}
	f_{X(1,n,\tilde{m},k), \; \ldots, \; X(n,n,\tilde{m},k)} \left( x_1, \; \ldots, \; x_n \right) 
		\quad\quad\quad\quad\quad\quad\quad\quad\quad\quad\quad\quad\quad  \\
	= k \left( \prod^{n-1}_{j=1} \gamma_j \right) \left( \prod^{n-1}_{i=1} (1-F(x_i))^{m_i} f(x_i) \right)
			(1-F(x_n))^{k-1} f(x_n)
\end{split}
\]

Let $X_{1:n}, X_{2:n}, \ldots, X_{n:n}$ denote the order statistics for random variables $X_1, X_2$, $\ldots, X_n$. Similarly, let $X^*_{1:n}, X^*_{2:n}, \ldots, X^*_{n:n}$ denote the generalized (or $m$-generalized) order statistics for an underlying distribution function $F$, when $m_i = m$ for all $1 \leq i \leq n-1$. For more details about generalized order statistics, the reader is referred to Kamps (1995).

\section{Some Contraction Type Characterizations}
\label{Contractions}

Beutner and Kamps (2008) considered characteriaztions of type (\ref{basic}), where $U \sim Pow(\alpha)$ or $U \sim Par(\alpha)$, and $X,Y$ are neighboring generalized order statistics. In particular, for $U \sim Pow(\alpha)$, they obtained characterizations based on the following relations:

	\[
		X^*_{r:n} \stackrel{d}= X^*_{r+1:n} \cdot U,
	\]
		
	\begin{equation} \label{beutner}
		X^*_{r:n} \stackrel{d}= X^*_{r:n-1} \cdot U,
	\end{equation}

	\[
		X^*_{r:n-1} \stackrel{d}= X^*_{r+1:n} \cdot U.
	\]

\vspace{0.2cm} 
The results in Beutner and Kamps (2008) generalize some other characterizations, given, for example, in Oncel et al. (2005), or Weso{\l}owski and Ahsanullah (2004). The main result of this paper is a characterization of type (\ref{basic}), which is similar to the relation given in the following result of Beutner and Kamps (2008).

	\begin{theorem} \label{BeutnerKamps2008}
		Let $X_{1:n-1}^{*}, \; \ldots , \; X_{n-1:n-1}^{*}$ and $X_{1:n}^{*}, \; \ldots , \; X_{n:n}^{*}$ 
		be $m$-generalized order statistics, $m \neq 1$, based on $F$, let $V$ be independent of 
		$X_{1:n-1}^{*}, \; \ldots , \; X_{n-1:n-1}^{*}$ and let $\alpha >0$ and 
		$r \in \left\{1, \; \ldots, \; n-1 \right\}$ be fixed. Then any two of the following three conditions 
		imply the third:
		
		\begin{enumerate}
				\item $V \sim Pow(\alpha)$,
    
				\item $  X^*_{r:n} \stackrel{d}= X^*_{r:n-1} V  $, 
				
				\item 
					\[
					  F(x) =  
						\begin{cases}
						  G_{m, \lambda, \beta}(x), 
					\: x \in \left(0, \left[ (m+1) \lambda \right]^{-\frac{1}{\beta}}\right), \textup{if} \; m>-1 \\
					    1 - \left( 1 + \lambda (-x)^{\beta} \right)^{-\frac{1}{m+1}}, \textup{if} \;  m<-1 
					  \end{cases}
					\]
		\end{enumerate}	
	\end{theorem}
	
In the next section we will present the main result, which in contrast to the above characterization, 
includes a product of two order statistics.

\section{Results}
\label{Results}

\noindent As a preparation for the main result, we introduce a preliminary distributional result, which involves products of maximal order statistics from a uniform distribution. 

Let $X_1, X_2, \ldots, X_n$ be iid random variables with continuos distribution function $F(x)$. 
If $X_i \sim Pow(1)$, that is, uniformly distributed on $(0,1)$, then it can be shown that the following relation is true:

	\[
		U_{1:n} \stackrel{d}{=} U^{(1)}_{1:1} U^{(2)}_{2:2} \cdots U^{(n)}_{n:n}
	\]

\noindent Here the $U^{(i)}_{i:i}$, $1 \leq i \leq n$, denote the maximum order statistics from independent samples of size $i$ from a uniform distribution on $(0,1)$. It should be noted that the $U_{i:i}$ in this product are all independent. In fact, the following result is true as well, which is given here without proof.

\begin{theorem} \label{arslan1}
	Let $X_1, \; X_2, \; ..., \; X_n, \; ...$ be iid absolutely continuous random variables with cdf $F(x), \; x \in (0,1)$. 
	If $\{ U^{(i)}_{i:i} \}$, $i \in \{ 1, \; 2, \; ..., \; n \}$ are independent samples of size $i$ from a uniform distribution on $(0,1)$, and
		
	\[ X_{1:n} \stackrel{d}{=} U^{(1)}_{1:1} U^{(2)}_{2:2} \cdots U^{(n)}_{n:n}, \]
	
\noindent	then $F(x) = x, \; x \in (0,1)$.
\end{theorem}

\begin{remark}
The proof of Theorem \ref{arslan1} can be shown directly with multiple integration but this result will also follow from the next theorem.
\end{remark}

Before stating the main results, it will be convenient to introduce ranked set samples (RSS). Ranked set sampling is an alternative sampling design to simple random sampling when actual measurement is either difficult or expensive, but ranking a few units in a small set is relatively easy and inexpensive.
This sampling design was first introduced by McIntyre (1952, 2005) \nocite{McIntyre1952}. 

An RSS can be simply described as follows. Let $X_1, \; X_2, \; \ldots , \; X_n, \; \ldots$ be independent and identically distributed random variables with cdf $F$.

			\vspace{0.3cm}
			\begin{math}
					\begin{array}{cccccc}
						\underline{X_{1:n}} & X_{2:n} & ... & X_{n:n} & \rightarrow & X_{[1,n]} \sim F_{1:n}(x) \\
						X_{1:n} & \underline{X_{2:n}} & ... & X_{n:n} & \rightarrow & X_{[2,n]} \sim F_{2:n}(x) \\
						... & ... & ... & ... & \rightarrow & ... \\
						X_{1:n} & X_{2:n} & ... & \underline{X_{n:n}} & \rightarrow & X_{[n,n]} \sim F_{n:n}(x)
					\end{array}		
			\end{math}

\vspace{0.3cm}
The following results can be interpreted as special cases of the above setup. In particular, we will start by considering the following sets of independent random variables \\

			\begin{math}
					\begin{array}{cccccccc}
						X_{1:n} 	& \cdots & \underline{X_{k:n}} 		& \cdots & X_{n-1:n} 	 & X_{n:n} & 
							\rightarrow & X_{[k,n]} \sim F_{k:n}(x) \\
						Y_{1:n-1} & \cdots & \underline{Y_{k:n-1}} 	& \cdots & Y_{n-1:n-1} & &
							\rightarrow & X_{[k,n-1]} \sim F_{k:n-1}(x) \\
						Z_{1:n} 	& \cdots & Z_{k:n} 								& \cdots & Z_{n-1:n} 	 & \underline{Z_{n:n}} &
							\rightarrow & X_{[n,n]} \sim F_{n:n}(x) 
					\end{array}		
			\end{math}
		
\begin{theorem}
  Let $\left\{ X_1, \ldots, X_n \right\}$, $\left\{ Y_1, \ldots, Y_{n-1} \right\}$, and 
  $\left\{ Z_1, \ldots, Z_n \right\}$ be independent sets of random variables with common 
  distribution function $F(x) = x^{\alpha}$, that is $X_i \sim Pow(\alpha)$. 
  Then, for any $1 \leq k \leq n-1$  
	\begin{equation}\label{relation}
		X_{k:n} \stackrel{d} = Y_{k:n-1} Z_{n:n}.
	\end{equation}
\end{theorem}

\begin{proof}
Let $Y = Y_{k:n-1} Z_{n:n}$. Then 

\begin{equation}
	\begin{split}
	  P \left( Y \leq y \right) & = \int_0^1 F_{n:n} \left( \frac{y}{u} \right) f_{k:n-1} \left( u  \right) du \\
	                            & = F_{k:n-1} \left( y \right) + \int_y^1 F_{n:n} \left( \frac{y}{u} \right) f_{k:n-1} \left( u \right) du
	\end{split}
\end{equation}

Hence, $X_{k:n} \stackrel{d} = Y_{k:n-1} Z_{n:n}$ implies that

\begin{equation}
	F_{k:n} (x) = F_{k:n-1} (x) + \int_x^1 F_{n:n} \left( \frac{x}{u} \right) f_{k:n-1} \left( u \right) du.
\end{equation}
\noindent Since for any $1 \leq k \leq n-1, \; n \left[ F_{k:n} - F_{k:n-1} \right] f = F f_{k:n}$, (see, for example, Weso{\l}owski and Ahsanullah (2004)\nocite{WesAhs04}), it follows that

\begin{equation}
	\frac{F(x) f_{k:n} (x)}{n f(x)} = \int_x^1 F_{n:n} \left( \frac{x}{u} \right) f_{k:n-1} \left( u \right) du.
\end{equation}

Evaluating the integral, using $X_i \sim Pow(\alpha)$,

\[ 
  \begin{split}
    \int_x^1 F_{n:n} \left( \frac{x}{u} \right) f_{k:n-1} \left( u \right) du
        & = k \binom{n-1}{k} \int_x^1 \left( \frac{x^{\alpha}}{u^{\alpha}} \right)^n 
	         (u^{\alpha})^{k-1} (1 - u^{\alpha})^{n-k-1} \alpha u^{\alpha-1} du  \\
				& = k \binom{n-1}{k} x^{\alpha n} \int_x^1 \left( u^{\alpha} \right)^{-n+k-1} (1 - u^{\alpha})^{n-k-1} u^{\alpha-1} du  \\
				& = k \binom{n-1}{k} x^{\alpha n} \int_{x^\alpha}^1 t^{-n+k-1} (1-t)^{n-k-1} dt,  \; \; (t=u^{\alpha}) \\
				& = k \binom{n-1}{k} x^{-\alpha n} \int_{x^\alpha}^1 (v - 1)^{n-k-1} dv  \\
				& = k \binom{n-1}{k} x^{\alpha n} (x^{-\alpha} - 1)^{(n-k)}
  \end{split}
\]

On the other hand, it follows that

\[
  \begin{split}
		\frac{F(x) f_{k:n} (x)}{n f(x)} & = k \binom{n-1}{k} \left[ F(x) \right]^k \left[ 1 - F(x) \right]^{n-k}  \\
																		& = k \binom{n-1}{k} ( x^\alpha )^k (1 - x^\alpha)^{n-k}
  \end{split}  
\]

\noindent Hence, the proof is complete.

\end{proof}

\vspace{0.3cm}
As an immediate consequence of this theorem, any order statistic $X_{k:n}$ from a power distribution can be expressed in terms of
maximum order statistics:

\begin{corollary} \label{cor}
  Let $X_1, \ldots, X_n$ be iid with cdf $F(x) = x^{\alpha}$, $x \in (0,1)$, that is $X_i \sim Pow(\alpha)$. 
  Then for any $1 \leq k \leq n-1$ 
	 \[ X_{k:n} \stackrel{d} = X_{[k:k]} X_{[k+1:k+1]} \cdots X_{[n:n]}.
   \]
\end{corollary}

\begin{remark}
Note that for $k=1$ and $\alpha = 1$, we obtain 
\[ U_{1:n} \stackrel{d}{=} U_{[1:1]} U_{[2:2]} \cdots U_{[n:n]}. \]
\end{remark}

This corollary can be interpreted in light of ranked set samples when it is considered as a special case of a ranked set sample as follows. Let $X_1, \; X_2, \; \ldots , \; X_n, \; \ldots$ be iid random variables with cdf $F$. Consider the following maxima scheme. \\

			\begin{math}
					\begin{array}{cccccc}
						\underline{X_{1:1}} &   &   &   & \rightarrow & X_{[1,1]} \sim F_{1:1}(x) \\
						X_{1:1} & \underline{X_{2:2}} &  &   & \rightarrow & X_{[2,2]} \sim F_{2:2}(x) \\
						... & ... & ... &  & \rightarrow & ... \\
						X_{1:n} & X_{2:n} & ... & \underline{X_{n:n}} & \rightarrow & X_{[n,n]} \sim F_{n:n}(x)
					\end{array}		
			\end{math}

\vspace{0.3cm}
\noindent If $F(x) = x^{\alpha}$, $x \in (0,1)$, that is $X_i \sim Pow(\alpha)$, then 
	 \[ X_{k:n} \stackrel{d} = X_{[k:k]} X_{[k+1:k+1]} \cdots X_{[n:n]}. \]

\vspace{0.3cm}
To prove the main result, we will use the following lemma. We note that this lemma can be considered as another special variant of the Choquet-Deny Theorem. The proof of this lemma is almost the same as the proof of Theorem 1 in Fosam and Shanbhag (1997) \nocite{FosamShanbhag97}. 

\begin{lemma}\label{lema}
Let $H$ be a nonnegative function that is not identically equal to zero on $A=(0,1)$. Also, let $\left\{ \mu_x: x \in A \right\}$ be a family of finite measures such that for each $x \in A$, $\mu (B_x)>0$, where $B_x = (x,1)$. Then a continuous real-valued function $H$ on $A$ such that $H(x)$ has a limit as $x$ tends to $1-$, satisfies
	\begin{equation} \label{eqn_lema}
		\int^{1}_{x} \left[ H(x) - H \left( \frac{x}{u} \right) \right] \mu_{x}(du) = 0, \ x \in (0,1),
	\end{equation}
	
\noindent if and only if it is identically equal to a constant.
\end{lemma}

\begin{remark}
We note that in this lemma $\left\{ \mu_x: x \in A \right\}$ is a family of finite measures. If these finite measures are probability measures, then equation (\ref{eqn_lema}) can be written as
\[
	H(x) = \int^{1}_{x} H \left( \frac{x}{u} \right) \mu_{x}(du), \ x \in (0,1).
\]
\end{remark}

\begin{theorem} \label{result}
  	Let $\left\{ X_1, \ldots, X_n \right\}$, $\left\{ Y_1, \ldots, Y_{n-1} \right\}$, and 
  	$\left\{ Z_1, \ldots, Z_n \right\}$ be independent sets of random variables with common 
  	distribution function $F$ such that the left and right extremities of $F$  equal to 0 and 1, 
  	respectively. In addition, assume that $f$ exists and is continuous on $(0,1)$. 
  	If for a fixed $1 \leq k \leq n-1$,	 
	 		\[ X_{k:n} \stackrel{d} = Y_{k:n-1} Z_{n:n}, \]
   
  \noindent then $X_i \sim Pow(\alpha)$.
\end{theorem}

\begin{proof}
$X_{k:n} \stackrel{d} = Y_{k:n-1} Z_{n:n}$ implies that

	\begin{equation} \label{eqn1}
		F_{k:n}(x) = F_{k:n-1}(x) + \int^{1}_{x} F_{n:n} \left( \frac{x}{u} \right) f_{k:n-1}(u) du
	\end{equation}

Since $n \left[ F_{k:n}(x) - F_{k:n-1}(x) \right] F(x) = n f(x)$, we have

	\begin{equation} \label{eqn2}
		F(x) f_{k:n}(x) = n f(x) \int^{1}_{x} F_{n:n} \left( \frac{x}{u} \right) f_{k:n-1}(u) du
	\end{equation}

\noindent By differentiating (\ref{eqn1}) with respect to $x$, it follows that

	\begin{equation} \label{eqn3}
		f_{k:n}(x) = \int^{1}_{x} f_{n:n} \left( \frac{x}{u} \right) \frac{1}{u} f_{k:n-1}(u) du
	\end{equation}

From (\ref{eqn2}) and (\ref{eqn3}), we obtain

	\[ n f(x) \int^{1}_{x} F_{n:n} \left( \frac{x}{u} \right) f_{k:n-1}(u) du =
			F(x) \int^{1}_{x} f_{n:n} \left( \frac{x}{u} \right) \frac{1}{u} f_{k:n-1}(u) du \]

\noindent or

	\[
		\int^{1}_{x} \left[ n f(x) F_{n:n} \left( \frac{x}{u} \right) - 
		F(x) f_{n:n} \left( \frac{x}{u} \right) \frac{1}{u} \right] f_{k:n-1}(u) du = 0, \ x \in (0,1)
	\]

\noindent This last equation can be written as

	\begin{equation}
		\int^{1}_{x} F^{n} \left( \frac{x}{u} \right) \left[ \frac{xf(x)}{F(x)}  - 
		 \frac{ \frac{x}{u} f \left( \frac{x}{u} \right) } {F \left( \frac{x}{u} \right)} \right] f_{k:n-1}(u) du = 0, \ x \in (0,1),
	\end{equation}

\noindent or, defining $H(x) = \frac{xf(x)}{F(x)}$, 

	\begin{equation}
		\int^{1}_{x} F^{n} \left( \frac{x}{u} \right) \left[ H(x) - 
		H \left( \frac{x}{u} \right) \right] f_{k:n-1}(u) du = 0, \ x \in (0,1).
	\end{equation}

Now, using Lemma \ref{lema} with $\mu_x (B) = \int_{B \cap B_x} F^{n} \left( \frac{x}{u} \right) f_{k:n-1}(u) du$, $B_x = (x,1)$, it follows that $H$ is constant on $(0,1)$;

	\begin{equation}\label{eqn}
		H(x) = \frac{xf(x)}{F(x)} = k, \; x \in (0,1),
	\end{equation}

\noindent for some $k \in \mathbf{R}$. The solution of this separable differential equation with boundary conditions $F(0)=0$ and
$F(1)=1$ implies that $F(x) = x^{\alpha}, \; x \in (0,1)$. 

\end{proof}

\section{Summary}
\label{sum}
A new characterization result based on products of order statistics from independent samples has been obtained.  This characterization result is interesting because in some sense it is an extension of contraction type results studied in the literature. In addition, it may be related to some special scheme of ranked set sampling. Another interesting point of the main result is that as a by-product a representation theorem for order statistics is obtained. Actually, the $k$-th order statistic from a power function distribution can be expressed as a product of independent maximum order statistics as shown in Corollary \ref{cor}. 

We also note that the proof of the main result is obtained by applying a new variant (Lemma \ref{lema}) of the Choquet-Deny Theorem (see, for example, Fosam and Shanbhag (1997) \nocite{FosamShanbhag97}). 

\bibliographystyle{plain}

\end{document}